\documentclass[12pt,reqno]{amsart}
\usepackage[top=2cm,bottom=2cm,right=2.5cm,left=2.5cm]{geometry}
\usepackage{amssymb}
\usepackage{amsmath, amsthm, amscd, amsfonts, amssymb, graphicx, color}
\usepackage[bookmarksnumbered, colorlinks, plainpages]{hyperref}
\hypersetup{colorlinks=true,linkcolor=red, anchorcolor=green, citecolor=cyan, urlcolor=red, filecolor=magenta, pdftoolbar=true}
 
\usepackage{hyperref}

\newtheorem{theorem}{Theorem}[section]
\newtheorem{lemma}[theorem]{Lemma}

\newtheorem{corollary}[theorem]{Corollary}
\theoremstyle{definition}
\newtheorem{definition}[theorem]{Definition}
\newtheorem{example}[theorem]{Example}

\theoremstyle{remark}
\newtheorem{remark}[theorem]{Remark}

\numberwithin{equation}{section}

\begin{document}

\setcounter{page}{1}

\title{$\varphi-$fixed point theorems in $C^{\ast}-$algebra valued partial metric spaces}

	\author{Hafida Massit$^{1*}$ and Mohamed Rossafi$^{2}$}

\address{$^{1}$Department of Mathematics, University of Ibn Tofail, B.P. 133, Kenitra, Morocco}
\email{\textcolor[rgb]{0.00,0.00,0.84}{massithafida@yahoo.fr}}
\address{$^{2}$LaSMA Laboratory Department of Mathematics Faculty of Sciences, Dhar El Mahraz University Sidi Mohamed Ben Abdellah, B. P. 1796 Fes Atlas, Morocco}
\email{\textcolor[rgb]{0.00,0.00,0.84}{rossafimohamed@gmail.com; mohamed.rossafi@usmba.ac.ma}}

\subjclass[2010]{Primary 47H10; Secondary 54H25.}

\keywords{$\varphi-$fixed point, $C^{\ast}$-algebra valued metric space, $C^{\ast}$-algebra valued partial metric space.}

\date{
\newline \indent $^{*}$Corresponding author}

\begin{abstract}
In this present article, we etablish some existence results of $\varphi-$fixed point of a mapping in a $C^{\ast}$-algebra valued metric spaces and we deduce some fixed point  theorems in $C^{\ast}$-algebra valued partial metric spaces. Non-trivial examples are further provided to support the hypotheses of our results.
\end{abstract}
\maketitle
\section{Introduction }
A $C^{\ast}$-algebra valued  metric spaces were introduced by Ma et al. \cite{Z.Ma} as a generalization of metric spaces they proved certain fixed point theorems, by giving the definition of $C^{\ast}$-algebra valued contractive mapping analogous to Banach contraction principle \cite{BAN}.

In this paper, inspired by the work done in \cite{Jeli}, we introduce the notion of $ \varphi -$fixed point results for various classes of operators defined on a $C^{\ast}$-algebra valued metric space $ (X,\mathbb{A},d) $. The obtained results are used to give some fixed point theorems, where $ X $ is endowed with a partial metric $ p $ . 
\section{preliminaries}
 Throughout this paper, we denote $  \mathbb{A}$ by an unital (i.e, unity element I) $C^{\ast}$-algebra with linear involution $\ast $, such that for all $ x,y \in \mathbb{A}  $ ,
\begin{center}
$ (xy)^{\ast} = y^{\ast}x^{\ast}$, and $ x^{\ast \ast} = x $.
\end{center}
We call an element $x\in\mathbb{A}$ a positive element, denoted by $x\succeq\theta$
 if $x\in\mathbb{A}_{h}=\lbrace x\in \mathbb{A}: x= x^{\ast}\rbrace $ and  $\sigma(x) \subset \mathbb{R}_{+}$, where  $\sigma(x)$  is the spectrum of $ x $. Using positive element, we can define a partial ordering $\preceq$ on $\mathbb{A}_{h}$ as follows : 
\begin{center}
$x\preceq y$  if and only if  $y-x\succeq\theta$
\end{center}  where $\theta$ means the zero element in $\mathbb{A}$. We denote the set ${x \in \mathbb{A} :x\succeq\theta}$ by $\mathbb{A}_{+ }$  and $\vert x \vert = (x^{\ast}x)^{\frac{1}{2}}$
\begin{remark}
When $  \mathbb{A}$ is a unital $C^{\ast}$-algebra,then for any $x\in \mathbb{A}_{+}$ we have \begin{center}
 $x\preceq I \Longleftrightarrow \Vert x \Vert \leq 1$
\end{center}
\end{remark}
\begin{definition} \cite{S.G}
 Let $X$ be a non-empty set. A mapping $\ p:X\times X\rightarrow \mathbb{A}$ is called a $C^{\ast}$-algebra valued metric on $ X $ if the following conditions are satisfied: 
 \begin{itemize}
 	\item[(i)]  $\theta \preceq p(x,y)$  for all $ x,y \in X $ and $ p(x,x) = p(y,y) = p(x,y)$ if and only if $x=y$ 
\item[(ii)]  $p(x,y)=p(y,x)$  for all $x,y\in X$. 
\item[(iii)] $ p(x,x) \preceq p(x,y) $ for all $x,y\in X$.
\item[(iv)]  $p(x,y)\preceq p(x,z)+p(z,y) -p(z,z)$  for all $ x,y,z \in X $.
\end{itemize}
Then $\left( X ,\mathbb{A}_{+} ,p\right) $ is called a $C^{\ast}$-algebra valued partial metric space.
\end{definition}
If we take $ \mathbb{A}= \mathbb{R} $, then the new notion of $C^{\ast}$-algebra valued partial metric space becomes equivalent to the definition of the real partial metric space.
	\begin{example} Let$ X=[0,1] $ and  $ x\in \mathbb{A} $ be a nonzero element.
		Define $ p(s,t)= max\{1+s,1+t\}xx^{\ast}$.
	Then we can easily show that $ p:X\times X \rightarrow \mathbb{A}$ is a $C^{\ast}$-algebra valued partial metric.
\end{example}
\begin{example} Let $ X=[0,1] $ and $ \mathbb{A}= \mathbb{R}^{2} $
	with a usual norm, be a real Banach space.
	Let $ p: X \times X \rightarrow \mathbb{R}^{2} $ be given as follows: 
	\begin{center}
		$ p(x,y)=( \vert x-y\vert ,\vert x-y \vert ) $
	\end{center}
	Then $ (X,\mathbb{R}^{2},p ) $ is a complete $C^{\ast}$-algebra valued partial metric.
\end{example}
\begin{definition} \cite{SCD}
 Let $\left( X ,\mathbb{A} ,p\right) $  be a  $C^{\ast}$-algebra valued partial metric space.
 \begin{itemize}
 	\item[(1)] A sequence $ \left\lbrace x_{n}\right\rbrace\subset X $ converges to $x \in X $ whenever for every $\varepsilon> 0$ there is a natural number $N$ such that for all $n>N$ ,
\begin{center}
$\Vert p(x_{n},x)+p (x,x)\Vert \leq \varepsilon$
\end{center} 
 We denote it by
 \begin{center}
	$lim_{n\longrightarrow \infty} p(x_{n}, x)-p(x,x)= \theta$.
\end{center}
\item[(2)]$ \{x_{n}\} $ is a partial Cauchy sequence respect to $ \mathbb{A} $, whenever $\varepsilon> 0$ there is a natural number $N$ such that 
\begin{center}
	$ (p(x_{n},x_{m}) -\dfrac{1}{2} p(x_{n},x_{n}) - \dfrac{1}{2} p(x_{m},x_{m})) ((p(x_{n},x_{m}) -\dfrac{1}{2} p(x_{n},x_{n}) - \dfrac{1}{2} p(x_{m},x_{m}))^{\ast}  \preceq \varepsilon^{2}$.
\end{center}
for all $n,m > N$
\item[(3]) $\left( X ,\mathbb{A}_{+} ,p\right) $ is said to be complete  with respect to $\mathbb{A}$  if every partial Cauchy sequence with respect to $\mathbb{A}$  converges to a point $ x $ in $ X $ such that
\begin{center}
	$ lim_{n \rightarrow \infty}(p(x_{n},x) -\dfrac{1}{2} p(x_{n},x_{n}) -\dfrac{1}{2}p(x,x))= \theta $.
\end{center}
From given $ C^{\ast} -$algebra-valued partial metric, we can obtain a $ C^{\ast} -$ algebra-valued metric. 

Put 
\begin{center}
	$ p^{s}(x,y)= 2p(x,y)-p(x,x) -p(y,y)$
\end{center}
\end{itemize}
 Then $ p^{s } $ is a $ C^{\ast}- $ algebra -valued metric.
\end{definition}
\begin{lemma} \cite{SCD}
	Let $\left( X ,\mathbb{A} ,p\right) $  be a $ C^{\ast} -$ algebra- valued partial metric space.
	\begin{itemize}
			\item[(1)] $ \{x_{n}\} $ is a partial Cauchy sequence in  $\left( X ,\mathbb{A} ,p\right) $ if and only if it is 
 Cauchy in the $ C^{\ast} -$ algebra -valued metric  $\left( X ,\mathbb{A} ,p^{s}\right) $.
\item[(2)]A  $ C^{\ast} -$ algebra- valued partial metric space  $\left( X ,\mathbb{A} ,p\right) $ is complete if and only if  $ C^{\ast} -$ algebra- valued metric space   $\left( X ,\mathbb{A} ,p^{s}\right) $ is complete. Furthermore,
\begin{center}
	$ lim_{n\rightarrow \infty}p^{s}(x_{n},x) = \theta \Leftrightarrow lim_{n \rightarrow \infty}(2p(x_{n},x)-p(x_{n},x_{n})-p(x,x)) =\theta $
\end{center}
or
 \begin{center}
	$  lim_{n\rightarrow \infty}p^{s}(x_{n},x) = \theta \Leftrightarrow lim_{n \rightarrow \infty}(p(x_{n},x)-p(x_{n},x_{n})) =\theta$  
	
	 and   $lim_{n \rightarrow \infty}(p(x_{n},x)-p(x,x)) =\theta $
\end{center}
\end{itemize}
\end{lemma}
\begin{lemma}\cite{SCD}
Assume that $ x_{n} \rightarrow x $ and $ y_{n} \rightarrow y $ as $ n \rightarrow \infty $ in a $ C^{\ast}- $
algebra valued partial metric space  $\left( X ,\mathbb{A} ,p\right) $.
Then 
\begin{center}
	$ lim_{n \rightarrow \infty} (p(x_{n},y_{n}) -p(x_{n},x_{n})) =p(x,y)-p(x,x) $
\end{center}
and
\begin{center}
		$ lim_{n \rightarrow \infty} (p(x_{n},y_{n}) -p(y_{n},y_{n})) =p(x,y)-p(y,y) $
\end{center}
\end{lemma}
\section{Main result}
\begin{definition} Let $ (X,\mathbb{A},d) $ be a $ C^{\ast}- $	algebra valued metric space 
	
	$ \varphi:X\rightarrow \mathbb{A}^{+}  $  be a given function, and $ T:X \rightarrow X $ be an operator.
	
	We denote by 
	
	$ T^{0}:= 1_{X} $ ,  $ T^{1}:= T $ ,  $ T^{n+1}:= T\mathsf{o} T ^{n} ; n \in \mathbb{N} $
	
	the iterate operators of $ T $.
	
	We denote $ F_{T}:= \{x\in X; Tx= x \} $ the set of all fixed points of  $ T $.
	
	and  $ Z_{\varphi}:= \{x\in X; \varphi(x)= \theta \} $ the set of all zeros of the function $ \varphi $ .
	
	An element $ z \in X $ is a $ \varphi- $ fixed point of the operator $ T $ if and only if $ z \in F_{T}\cap Z_{\varphi} $.
	\end{definition}
\begin{definition} 
	We say that $ T $ is a $ \varphi- $ Picard operator if and only if
	\begin{itemize}
		\item [(i)] $  F_{T}\cap Z_{\varphi}= \{z\} $.
		\item [(ii)] $ T^{n}x \rightarrow z $ as $ n\rightarrow \infty $, $ \forall x\in X $
	\end{itemize} 
\end{definition}
\begin{definition} $ T $  said to be a weakly $ \varphi- $ Picard operator if and only if
	\begin{itemize}
		\item [(i)] $  F_{T}\cap Z_{\varphi} \not= \O $.
		\item [(ii)] the sequence  $ T^{n}x \rightarrow z $ converges for each $  x\in X $ and the limit is a $ \varphi- $ fixed point of the operator $ T $. 
	\end{itemize} 
	\end{definition}
\begin{definition}Let $\mathcal{F} $ the set of functions $ F: \mathbb{A}_{+}^{3}  \rightarrow  \mathbb{A}_{+}$ satisfying the following conditions:
	\begin{itemize}
		\item [(F1)] $ max\{x,y\} \preceq F(x,y,z) $ for all $ x,y,z \in \mathbb{A}_{+} $
		\item [(F2)] $ F(\theta, \theta, \theta) = \theta$
		\item [(F3)] $ F $ is continuous.
		\end{itemize}
\end{definition}
\begin{example}Let $ \mathbb{A}_{+}=\mathbb{R}^{2}_{+} $ and  $ F: \mathbb{A}_{+}^{3} \rightarrow \mathbb{A}_{+} $ such that
	
	 $ F(x,y,z)=x+ y+z $ , we have $ F \in \mathcal{F} $.
	\end{example}
\begin{definition} 
	Let $ (X,\mathbb{A},d) $ be a $C^{\ast}$-algebra valued metric space,
	
	 $ \varphi :X \rightarrow \mathbb{A}^{+} $ be a function, and $ F \in \mathcal{F} $.
	
The operator $ T: X \rightarrow X $ is said an $ (F,\varphi)- $ contraction with respect to the metric $ d $ if and only if 	
\begin{center}
	$ F(d(Tx,Ty), \varphi(Tx),\varphi(Ty)) \preceq k F(d(x,y),\varphi(x),\varphi(y))$  $ x,y\in X $, $ k\in (0,1) $ $(1)  $.
\end{center}
\end{definition}
\begin{theorem} 
	Let $ (X,\mathbb{A},d) $ be a complete $C^{\ast}$-algebra valued metric space,
	
	 $ \varphi :X \rightarrow \mathbb{A}^{+} $ be a function, and $ F \in \mathcal{F} $. Suppose that the following condictions hold:
	\begin{itemize}
		\item [(a)] $ \varphi $ is lower semi-continous
		\item[(b)] $ T: X \rightarrow X $ is an $ (F,\varphi)- $ contraction with respect to the metric $ d $ 
	\end{itemize}
	Then 
		\begin{itemize}
		\item [(i)] $ F_{T} \subseteq Z_{\varphi} $
		\item[(ii)] $ T $ is a $ \varphi- $ Picard operator
		\item [(iii)] $ \forall x\in X $, $ \forall n \in \mathbb{N} $
		
		$ d(T^{n}x,z)  \preceq \dfrac{k^{n}}{1-k} F(d(Tx,x),\varphi(Tx),\varphi(x))$ where $ \{z\} \in F_{T} \cap Z_{\varphi}= F_{T} $
	\end{itemize}

\end{theorem}
\begin{proof}Let $ u\in X $ be a fixed point of $ T $. By $ (1) $  with $ x=y=u $ we obtain
	
	 $ F(\theta,\varphi(u),\varphi(u)) \preceq k F(\theta, \varphi(u),\varphi(u))$ 
	$ \Rightarrow F(\theta, \varphi(u),\varphi(u))= \theta $ $ (k\in (0,1)) $ $ (2) $
	
	From $ (F1) $ we have
	
	    $ \varphi(u) \preceq F(\theta ,\varphi(u),\varphi(u)) $  $ (3) $
	 
	 By $ (2) $ and $ (3) $ we obtain $ \varphi(u)=\theta $ , which proves $ (i) $
	 
	 Let $ x\in X $ by $ (1) $ we have 
	 
	 	$ F(d(T^{n+1}x,T^{n}x), \varphi(T^{n+1}x),\varphi(T^{n}x)) \preceq k F(d(T^{n}x,T^{n-1}y),\varphi(T^{n}x),\varphi(T^{n}x))$  $ n \in \mathbb{N} \cup \{0\} $
	 	\begin{center}
	 	$ \Rightarrow F(d(T^{n+1}x,T^{n}x), \varphi(T^{n+1}x),\varphi(T^{n}x))  \preceq k^{n} F(d(Tx,x),\varphi(Tx), \varphi(x)) $ $ n \in \mathbb{N} \cup \{0\} $
	 	\end{center}
 	By $ F1 $ we have 
 	
 	$ max\{d(T^{n+1}x,T^{n}x),\varphi(T^{n+1}x)\} \preceq k^{n} F(d(Tx,x),\varphi(Tx), \varphi(x)) $
 	 $ n \in \mathbb{N} \cup \{0\}  $ $ (4) $
 	
 	$ \Rightarrow d(T^{n+1}x,T^{n}x)\preceq k^{n} F(d(Tx,x),\varphi(Tx), \varphi(x)) $$ n \in \mathbb{N} \cup \{0\}  $
 	
 	Since $ k\in (0,1) $ wich implies that $ \{T^{n}x\}\ $ is a Cauchy sequence. By the completeness of $ (X,\mathbb{A},d) $, there exists a $ z\in X $ such 
 	
 	$ lim_{n\rightarrow \infty}d(T^{n}x,z)= \theta $ $ (5) $
 	
 	From $ (5) $ we have $ lim_{n\rightarrow \infty}\varphi(T^{n+1}x) = \theta$ $ (6) $
 	
 	Using that $ \varphi  $ is semi -continuous, and  from $ (5) $ and $ (6) $ we obtain $ \varphi(z)= \theta $
 	
 	Then
 	
 		$ F(d(T^{n+1}x,Tz), \varphi(T^{n+1}x),\varphi(Tz)) \preceq k F(d(T^{n}x,z),\varphi(T^{n}x),\varphi(z))$  , $ n\in \mathbb{N} \cup \{0\} $ $ (7) $
  	
 	Letting $ n\rightarrow \infty $ in $ (7) $ , using $  (5)$ ,$ (6)  $ , $ F2 $ and the continuity of $ F $ we have 
 	
 	$ F(d(z,Tz),\theta, \varphi (Tz)) \preceq k F(\theta,\theta,\theta)= \theta $.
 	
 	From $ F1 $ we obtain $ d(z,Tz)=\theta $ i.e $ z $ is a $ \varphi- $ fixed point of $ T $.

	Let $ v \in X $ another $ \varphi- $ fixed point of $ T $.
	
 	Putting $ x=z $ and $ y= v $ in $ (1) $ we have
 	 \begin{center}
 	 $ F(d(z,v),\theta, \theta)) \preceq k F(d(z,v),\theta,\theta)$  
 	
 	$ \Rightarrow d(z,v)= \theta$
 	
 		$ \Rightarrow z=v $    so we get $ (ii) $
 \end{center}	
 	Using $ (4) $ we get 
 	
 	$ d(T^{n}x,T^{n+m}x)  \preceq \dfrac{k^{n}(1-k^{m})}{1-k} F(d(Tx,x),\varphi(Tx),\varphi(x))$ $ n,m\in \mathbb{N} \cup \{0\}   $.
 	
 	Letting $ m\rightarrow \infty $ in the above inequality and using $ (5) $ we get
 	
 		$ d(T^{n}x,z)  \preceq \dfrac{k^{n}}{1-k} F(d(Tx,x),\varphi(Tx),\varphi(x))$ $ n\in \mathbb{N} \cup \{0\}   $. 	
\end{proof}
\begin{example} Let $ X =[0,1] $ and $ \mathbb{A}= \mathbb{R}^{2} $ .
	
	Define $ d:X\times X \rightarrow \mathbb{R}^{2}$ such that $ d(x,y)= (0,x+y) $ $ \forall x,y \in X $ , then $ (X,\mathbb{R}^{2},d) $ is a $ C^{\ast}- $ algebra valued metric space. 
	
	Let $ \varphi:X \rightarrow  \mathbb{R}^{2}$ such that $ \varphi(x)= (x,x) $
	
	and $ F:\mathbb{R}^{2}\times \mathbb{R}^{2} \times \mathbb{R}^{2} \rightarrow \mathbb{R}^{2}$ such that $ F(a,b,c) =a+b+c $ $ \forall a,b,c \in \mathbb{R}^{2} $
	
	Let $ T:X \rightarrow X $ such that $ Tx= \dfrac{x}{2} $.
	
	we have 
	$ F(d(Tx,Ty),\varphi(Tx),\varphi(Ty))= (\dfrac{x+y}{2}, x+y) $
	
and	$ F(d(x,y),\varphi(x),\varphi(y))= (x+y, 2(x+y) )$
	
	Then $ T $ has the required propreties montioned in theorem $ 3.7 $ so we obtain that $ 0 $ is a $ \varphi- $ fixed point of $ T $.
\end{example}
\begin{remark} 
	Taking $ \varphi=\theta $ and $ F(a,b,c)=a+b+c $ in Theorem $ 3.7 $ we get the Banach contraction principle in $C^{\ast}$-algebra valued metric space.
\end{remark}
	\begin{definition} 
		Let $ (X,\mathbb{A},d) $ be a $C^{\ast}$-algebra valued metric space, $ \varphi :X \rightarrow \mathbb{A}^{+} $ be a function, and $ F \in \mathcal{F} $
		  
		   $ T: X \rightarrow X $  said to be a graphic $ (F,\varphi)- $ contraction with respect to the metric $ d $ if and only if 	
		   
		   $ F(d(T^{2}x,Tx),\varphi(T^{2}x),\varphi(Tx)) \preceq kF(d(Tx,x),\varphi(Tx),\varphi(x)) $ $ x\in X $ $ k \in (0,1) $
	\end{definition}
\begin{theorem} 
	Let $ (X,\mathbb{A},d) $ be a complete $C^{\ast}$-algebra valued metric space, 
		
		$ \varphi :X \rightarrow \mathbb{A}^{+} $ be a function, and $ F \in \mathcal{F} $. Suppose that the following condictions hold:
		\begin{itemize}
			\item [(a)] $ \varphi $ is lower semi-continous
			\item [(b)] $ T: X \rightarrow X $ is a graphic $ (F,\varphi)- $ contraction with respect to the metric $ d $
			\item [(c)] $ T $ is continuous 
		\end{itemize}
		Then 
		\begin{itemize}
			\item [(i)] $ F_{T} \subseteq Z_{\varphi} $
			\item [(ii)] $ T $ is a weakly $ \varphi- $ Picard operator
			\item [(iii )] $ \forall x\in X $, if $ T^{n}x \rightarrow z $ as $  n\rightarrow \infty $
			
			then
			
			$ d(T^{n}x,z)  \preceq \dfrac{k^{n}}{1-k} F(d(Tx,x),\varphi(Tx),\varphi(x))$  $ n\in \mathbb{N} $ $ (8) $
		\end{itemize}
		\end{theorem}
\begin{proof}Let $ u\in X $ be a fixed point of $ T $. Applying $ (8) $  with $ x=u $, we obtain
	
	$ F(\theta,\varphi(u),\varphi(u)) \preceq k F(\theta, \varphi(u),\varphi(u))$
	 
	$ \Rightarrow F(\theta, \varphi(u),\varphi(u))= \theta $ $ (k\in (0,1)) $ $ (9) $
	
	From $ (F1) $ we have
	
	$ \varphi(u) \preceq F(\theta ,\varphi(u),\varphi(u)) $  $ (10) $
	
	By $ (9) $ and $ (10) $ we obtain $ \varphi(u)=\theta $ , which proves $ (i) $
	
	Let $ x\in X $ by $ (8) $ we have 
	
	$ F(d(T^{n+1}x,T^{n}x), \varphi(T^{n+1}x),\varphi(T^{n}x)) \preceq k F(d(T^{n}x,T^{n-1}x),\varphi(T^{n}x),\varphi(T^{n}x))$  $ n \in \mathbb{N} \cup \{0\} $
	\begin{center}
		$ \Rightarrow F(d(T^{n+1}x,T^{n}x), \varphi(T^{n+1}x),\varphi(T^{n}x))  \preceq k^{n} F(d(Tx,x),\varphi(Tx), \varphi(x)) $ $ n \in \mathbb{N} \cup \{0\} $
	\end{center}
	By $ F1 $ we have 
	
	$ max\{d(T^{n+1}x,T^{n}x),\varphi(T^{n+1}x)\} \preceq k^{n} F(d(Tx,x),\varphi(Tx), \varphi(x)) $ $ n \in \mathbb{N} \cup \{0\}  $ $ (11) $
	
	$ \Rightarrow d(T^{n+1}x,T^{n}x)\preceq k^{n} F(d(Tx,x),\varphi(Tx), \varphi(x)) $$ n \in \mathbb{N} \cup \{0\}  $
	
	Since $ k\in (0,1) $ which implies that $ \{T^{n}x\} $ is a Cauchy sequence. By the completeness of $ (X,\mathbb{A},d )$, there exists a $ z\in X $ such 
	
	$ lim_{n\rightarrow \infty}d(T^{n}x,z)= \theta $ $ (12) $
	
	From $ (12) $ we have $ lim_{n\rightarrow \infty}\varphi(T^{n+1}x) = \theta$ $ (13) $
	
	Using that $ \varphi  $ is lower semi -continuous, and  from $ (12) $ and $ (13) $ we obtain $ \varphi(z)= \theta $
	
	Then
	
	$ F(d(T^{n+1}x,Tz), \varphi(T^{n+1}x),\varphi(Tz)) \preceq k F(d(T^{n}x,z),\varphi(T^{n}x),\varphi(z))$  , $ n\in \mathbb{N} \cup \{0\} $. $ (14) $
	
	Letting $ n\rightarrow \infty $ in $ (14) $ , using $  (12)$ , $ (13)  $ , $ F2 $ and the continuity of $ F $ we have 
	
	$ F(d(z,Tz),\theta, \varphi (Tz)) \preceq k F(\theta,\theta,\theta)= \theta $.
	
	From $ F1 $ we obtain $ d(z,Tz)=\theta $ i.e $ z $ is a $ \varphi- $ fixed point of $ T $.
	
	Let $ v \in X $ another $ \varphi- $ fixed point of $ T $.
	
	Putting $ x=z $ and $ y= v $ in $ (8) $ we have
		\begin{center}
		$ F(d(z,v),\theta, \theta)) \preceq k F(d(z,v),\theta,\theta)$  
		
		$ \Rightarrow d(z,v)= \theta$

		$ \Rightarrow z=v $    so we get $ (ii) $
	\end{center}	
	Using $ (11) $ we get 
	
	$ d(T^{n}x,T^{n+m}x)  \preceq \dfrac{k^{n}(1-k^{m})}{1-k} F(d(Tx,x),\varphi(Tx),\varphi(x))$ $ n,m\in \mathbb{N} \cup \{0\}   $
	
	letting $ m\rightarrow \infty $ in the above inequality and using $ (12) $ we get
	
	$ d(T^{n}x,z)  \preceq \dfrac{k^{n}}{1-k} F(d(Tx,x),\varphi(Tx),\varphi(x))$ $ n\in \mathbb{N} \cup \{0\}   $
	\end{proof}
\begin{definition} 
	Let $ (X,\mathbb{A},d) $ be a $C^{\ast}$-algebra valued metric space,
	
	 $ \varphi :X \rightarrow \mathbb{A}^{+} $ be a function, and $ F \in \mathcal{F} $.
	
	The operator $ T: X \rightarrow X $  said to be an $ (F,\varphi)- $   weak contraction with respect to the metric $ d $ if and only if 	
	\begin{center}
		$ F(d(Tx,Ty), \varphi(Tx),\varphi(Ty)) \preceq k F(d(x,y),\varphi(x),\varphi(y)) + \alpha( F(d(y,Tx),\varphi(y),\varphi(Tx)) -F(\theta,\varphi(y),\varphi(Tx))) $  $ x,y\in X $,$ k\in (0,1) ,\alpha \geq 0 $ $(15)  $.
	\end{center}
	\end{definition}
\begin{theorem} Let $ (X,\mathbb{A},d) $ be a complete $C^{\ast}$-algebra valued metric space,
		
		 $ \varphi :X \rightarrow \mathbb{A}^{+} $ be a function, and $ F \in \mathcal{F} $.
		 
		  Suppose that the following condictions hold:
	\begin{itemize}
		\item [(a)] $ \varphi $ is lower semi-continous
		\item[(b)] $ T: X \rightarrow X $ is an  $ (F,\varphi)- $ weak contraction with respect to the metric $ d $
		\end{itemize}
	Then 
	\begin{itemize}
		\item [(i)] $ F_{T} \subseteq Z_{\varphi} $
		\item[(ii)] $ T $ is a weakly $ \varphi- $ Picard operator
		\item [(iii )] $ \forall x\in X $, if $ T^{n}x \rightarrow z $ as $  n\rightarrow \infty $
		
		then
		
		$ d(T^{n}x,z)  \preceq \dfrac{k^{n}}{1-k} F(d(Tx,x),\varphi(x),\varphi(Tx)) $ 
		 $ n\in \mathbb{N} $ 
	\end{itemize}
	\end{theorem}
\begin{proof}Let $ u\in X $ be a fixed point of $ T $. By $ (15) $  with $ x=y =u $ we obtain
	
	$ F(\theta,\varphi(u),\varphi(u)) \preceq k F(\theta, \varphi(u),\varphi(u))+\alpha( F(\theta,\varphi(u),\varphi(u))-F(\theta,\varphi(u),\varphi(u))$ 
\begin{center}	
	 $= k F(\theta,\varphi(u),\varphi(u)) $
	 
	$ \Rightarrow F(\theta, \varphi(u),\varphi(u))= \theta $ $ (k\in (0,1)) $ $ (16) $
\end{center}
	From $ (F1) $ we have
	
	$ \varphi(u) \preceq F(\theta ,\varphi(u),\varphi(u)) $  $ (17) $
	
	By $ (16) $ and $ (17) $ we obtain $ \varphi(u)=\theta $ , which proves $ (i) $
	
	Let $ x\in X $ by $ (15) $ we have 
	
	$ F(d(T^{n}x,T^{n+1}x), \varphi(T^{n}x),\varphi(T^{n+1}x)) \preceq k F(d(T^{n-1}x,T^{n}x),\varphi(T^{n-1}x),\varphi(T^{n}x)) +\alpha (F( \theta, \varphi(T^{n}x),\varphi(T^{n}x)))-F(\theta, \varphi(T^{n}x),\varphi(T^{n}x))$ 
	
	 $ n \in \mathbb{N} \cup \{0\} $
	\begin{center}
		$ \Rightarrow F(d(T^{n}x,T^{n+1}x), \varphi(T^{n+1}x),\varphi(T^{n}x))  \preceq k^{n} F(d(x,Tx),\varphi(x), \varphi(Tx)) $ $ n \in \mathbb{N} \cup \{0\} $
	\end{center}
	By $ F1 $ we have 
	
	$ max\{d(T^{n+1}x,T^{n}x),\varphi(T^{n+1}x)\} \preceq k^{n} F(d(Tx,x),\varphi(Tx), \varphi(x)) $ $ n \in \mathbb{N} \cup \{0\}  $ $ (18) $
	
	$ \Rightarrow d(T^{n+1}x,T^{n}x)\preceq k^{n} F(d(Tx,x),\varphi(Tx), \varphi(x)) $$ n \in \mathbb{N} \cup \{0\}  $
	
	Since $ k\in (0,1) $ which implies that $ \{T^{n}x\}\ $ is a Cauchy sequence. By the completeness of $ (X,\mathbb{A},d )$, there exists a $ z\in X $ such 
	
	$ lim_{n\rightarrow \infty}d(T^{n}x,z)= \theta $ $ (19) $
	
	From $ (19) $ we have $ lim_{n\rightarrow \infty}\varphi(T^{n+1}x) = \theta$ $ (20) $
	
	Using that $ \varphi  $ is lower semi -continuous, and  from $ (19) $ and $ (20) $ we obtain $ \varphi(z)= \theta $
	
	Then
	
	$ F(d(T^{n+1}x,Tz), \varphi(T^{n+1}x),\varphi(Tz)) \preceq k F(d(T^{n}x,z),\varphi(T^{n}x),\varphi(z))$  , $ n\in \mathbb{N} \cup \{0\} $ $ (21) $
	
	Letting $ n\rightarrow \infty $ in $ (21) $ , using $  (19)$ , $ (20)  $ , $ F2 $ and the continuity of $ F $ we have 
	
	$ F(d(z,Tz),\theta, \varphi (Tz)) \preceq k F(\theta,\theta,\theta)= \theta $.
	
	From $ F1 $ we obtain $ d(z,Tz)=\theta $ i.e $ z $ is a $ \varphi- $ fixed point of $ T $.
	
	Let $ v \in X $ another $ \varphi- $ fixed point of $ T $.
	
	Putting $ x=z $ and $ y= v $ in $ (15) $ we have
	\begin{center}
		$ F(d(z,v),\theta, \theta)) \preceq k F(d(z,v),\theta,\theta)$  
		
		$ \Rightarrow d(z,v)= \theta$
		
		$ \Rightarrow z=v $    so we get $ (ii) $
	\end{center}	
	Using $ (11) $ we get 
	
	$ d(T^{n}x,T^{n+m}x)  \preceq \dfrac{k^{n}(1-k^{m})}{1-k} F(d(Tx,x),\varphi(Tx),\varphi(x))$ $ n,m\in \mathbb{N} \cup \{0\}   $.
	
	Letting $ m\rightarrow \infty $ in the above inequality and using $ (12) $ we get
	
	$ d(T^{n}x,z)  \preceq \dfrac{k^{n}}{1-k} F(d(Tx,x),\varphi(Tx),\varphi(x))$ $ n\in \mathbb{N} \cup \{0\}   $.
\end{proof}
\begin{example}
	Let $ X=\mathbb{C} $ and $ \mathbb{A}= \mathbb{M}_{2}(\mathbb{C}) $ the set of all $ n \times n $ matrices with entries in $ \mathbb{C} $.
	
We define $ d(a,b)=  
		 \begin{pmatrix}
		\begin{array}{cc}
		\vert	a_{1}-b_{1}\vert& 0 \\ 
			0 & \vert a_{2}-b_{2}\vert
		\end{array} 
	\end{pmatrix}  $ 
	
	 $ F: \mathbb{M}_{2}(\mathbb{C}) \times \mathbb{M}_{2}(\mathbb{C}) \times \mathbb{M}_{2}(\mathbb{C}) \rightarrow \mathbb{M}_{2}(\mathbb{C}) $ $ F(A,B,C) = A^{2}+B+C$
	 
	 and $ \varphi : X\rightarrow \mathbb{M}_{2}(\mathbb{C}) $ with $ \varphi(x_{1},x_{2}) =    
	 \begin{pmatrix}
	 	\begin{array}{cc}
	 		2\vert	x_{1}-x_{2}\vert& 0 \\ 
	 		0 & 2\vert x_{1}-x_{2}\vert
	 	\end{array} 
	 \end{pmatrix}  $
 
 and $ T:X \rightarrow X $ defined by $ T(x_{1},x_{2}) = (\dfrac{x_{1}}{2},\dfrac{x_{2}}{2})$
 
 We have $ T $ satisfying the condition $ (15) $ with $ k=\frac{1}{2} $ and $ \alpha=4 $.
 
  Therefore, $ T $ has a $ \varphi- $ fixed point $ z=(0,0) $ .
	 \end{example}
\begin{theorem} Let $ (X,\mathbb{A},d) $ be a complete  $C^{\ast}$-algebra valued metric space,
	
	 $ \varphi :X \rightarrow \mathbb{A}^{+} $ be a function, and $ F \in \mathcal{F} $.
	 
	  Suppose that the following condictions hold:
	\begin{itemize}
		\item [(a)] $ \varphi $ is lower semi-continous
		\item[(b)] $ T: X \rightarrow X $ satisfying 
		
		$ F(d(Tx,Ty),\varphi(Tx),\varphi(Ty)) \preceq k[ F(d(Tx,x),\varphi(Tx),\varphi(x)) +F(d(Ty,y),\varphi(Ty),\varphi(y))] $ $ k\in(0,\frac{1}{2}) $ $ \forall x,y \in X $ $ (21) $
			\end{itemize}
	Then 
	\begin{itemize}
		\item [(i)] $ F_{T} \subseteq Z_{\varphi} $
		\item[(ii)] $ T $ is a $ \varphi- $ Picard operator
	\end{itemize}
	\end{theorem}
\begin{proof}Let $ u\in X $ be a fixed point of $ T $. By $ (21) $  with $ x=y=u $ we obtain
	
	$ F(\theta,\varphi(u),\varphi(u)) \preceq k[ F(\theta, \varphi(u),\varphi(u))+ F(\theta, \varphi(u),\varphi(u))]$ $ (k\in (0,\frac{1}{2})) $

	$ \Rightarrow F(\theta, \varphi(u),\varphi(u))= \theta $  
	
		From $ (F1) $ we have
	
	$ \varphi(u) \preceq F(\theta ,\varphi(u),\varphi(u)) $ $ \Rightarrow \varphi(u)= \theta $
	
	which we proves $ (i) $.
	
		Let $ x\in X $ be an arbitrary point, we have 
	
	$ F(d(T^{n+1}x,T^{n}x), \varphi(T^{n+1}x),\varphi(T^{n}x)) \preceq k [ F(d(T^{n+1}x,T^{n}x),\varphi(T^{n+1}x),\varphi(T^{n}x)) + F(d(T^{n}x,T^{n-1}x) , \varphi(T^{n}x),\varphi(T^{n-1}x))]$ 
	
	$ n \in \mathbb{N} \cup \{0\} $
	
	$ \Rightarrow  F(d(T^{n+1}x,T^{n}x), \varphi(T^{n+1}x),\varphi(T^{n}x)) \preceq \dfrac{k}{1-k} F(d(T^{n}x,T^{n-1}x) , \varphi(T^{n}x),\varphi(T^{n-1}x))$
	\begin{center}
		$ \Rightarrow F(d(T^{n+1}x,T^{n}x), \varphi(T^{n+1}x),\varphi(T^{n}x))  \preceq (\dfrac{k}{1-k})^{n} F(d(Tx,x),\varphi(Tx), \varphi(x)) $ $ n \in \mathbb{N} \cup \{0\} $
	\end{center}
	By $ F1 $ we have 
	
	$ max\{d(T^{n+1}x,T^{n}x),\varphi(T^{n+1}x)\} \preceq( \dfrac{k}{1-k})^{n} F(d(Tx,x),\varphi(Tx), \varphi(x)) $ $ n \in \mathbb{N} \cup \{0\}  $ $ (22)$
	
	$ \Rightarrow d(T^{n+1}x,T^{n}x)\preceq (\dfrac{k}{1-k})^{n} F(d(Tx,x),\varphi(Tx), \varphi(x)) $$ n \in \mathbb{N} \cup \{0\}  $
	
	Since $ k\in (0,\dfrac{1}{2}) $ which implies that $ \{T^{n}x\}\ $ is a Cauchy sequence. By the completeness of $ (X,\mathbb{A},d )$, there exists a $ z\in X $ such 
	
	$ lim_{n\rightarrow \infty}d(T^{n}x,z)= \theta $ $ (23) $
	
	From $ (23) $ we have $ lim_{n\rightarrow \infty}\varphi(T^{n+1}x) = \theta$ $ (24) $
	
	Using that $ \varphi  $ is lower semi -continuous, and  from $ (23) $ and $ (24) $ we obtain $ \varphi(z)= \theta $
	
	Then
	
	$ F(d(T^{n+1}x,Tz), \varphi(T^{n+1}x),\varphi(Tz)) \preceq k[ F(d(T^{n+1}x,T^{n}x),\varphi(T^{n+1}x),\varphi(T^{n}x))+F(d(Tz,z),\varphi(Tz),\varphi(z))$  , $ n\in \mathbb{N} \cup \{0\} $. $ (25) $
	
	Letting $ n\rightarrow \infty $ in $ (25) $ , using $  (23)$ , $ (24)  $ , $ F2 $ and the continuity of $ F $ we have 
	
	$(1-k) F(d(Tz,z),\theta, \varphi (Tz)) \preceq k F(\theta,\theta,\theta)= \theta $.
	
	From $ F1 $ we obtain $ d(z,Tz)=\theta $ i.e $ z $ is a $ \varphi- $ fixed point of $ T $.
	
	Let $ v \in X $ another $ \varphi- $ fixed point of $ T $.
	
	Putting $ x=z $ and $ y= v $ in $ (21) $ we have
	
	\begin{center}
		$ F(d(z,v),\theta, \theta)) \preceq k F(\theta,\theta,\theta) $  
		
		$ \Rightarrow d(z,v)= \theta$
		
		$ \Rightarrow z=v $    so we get $ (ii) $
	\end{center}	
	\end{proof}
\begin{remark} If we take  $ \varphi\equiv \theta $ and $ F(a,b,c)= a+b+c $ in Theorem $ 3.14 $ we obtain Kannan's fixed point on  $ C^{\ast} $ algebra valued metric space.
	 \end{remark}
\begin{theorem} Let $ (X,\mathbb{A},d) $ be a complete  $C^{\ast}$-algebra valued metric space, $ \varphi :X \rightarrow \mathbb{A}^{+} $ be a function, and $ F \in \mathcal{F} $. Suppose that the following condictions hold:
	\begin{itemize}
		\item [(a)] $ \varphi $ is lower semi-continous
		\item[(b)] $ T: X \rightarrow X $ satisfying 
		
		$ F(d(Tx,Ty),\varphi(Tx),\varphi(Ty)) \preceq \alpha F(d(x,y),\varphi(x),\varphi(y)) + \beta F(d(x,Tx),\varphi(x),\varphi(Tx)) +\gamma F(d(y,Ty),\varphi(y),\varphi(Ty))$ 
		
		 $ \alpha,\beta,\gamma \in[0,\infty) $  with $ \alpha +\beta +\gamma <1 $ and $ \forall x,y \in X $ $ (26) $
		\end{itemize}
	Then 
	\begin{itemize}
		\item [(i)] $ F_{T} \subseteq Z_{\varphi} $
		\item[(ii)] $ T $ is a $ \varphi- $ Picard operator
	\end{itemize}
	\end{theorem}
\begin{proof}Let $ u\in X $ be a fixed point of $ T $. Using $ (26) $  with $ x= y=u $ we obtain
	
	$ F(\theta,\varphi(u),\varphi(u)) \preceq \alpha F(\theta, \varphi(u),\varphi(u))+\beta( F(\theta,\varphi(u),\varphi(u))+\gamma F(\theta,\varphi(u),\varphi(u))$ 
	\begin{center}	
		$= (\alpha+\beta+\gamma) F(\theta,\varphi(u),\varphi(u)) $
		
		$ \prec F(\theta,\varphi(u),\varphi(u)) $
		
		$ \Rightarrow F(\theta, \varphi(u),\varphi(u))= \theta $  $ (27) $
	\end{center}
	From $ (F1) $ we have
	
	$ \varphi(u) \preceq F(\theta ,\varphi(u),\varphi(u)) $  $ (28) $
	
	By $ (27) $ and $ (28) $ we obtain $ \varphi(u)=\theta $ , which proves $ (i) $
	
	Let $ x\in X $ by $ (26) $ we have 
	
	$ F(d(T^{n}x,T^{n+1}x), \varphi(T^{n}x),\varphi(T^{n+1}x)) \preceq \alpha F(d(T^{n-1}x,T^{n}x),\varphi(T^{n-1}x),\varphi(T^{n}x)) +\beta (F( d(T^{n}x, T^{n+1}), \varphi(T^{n}x),\varphi(T^{n+1}x)))+\gamma F(d(T^{n-1}x,T^{n}x)), \varphi(T^{n-1}x),\varphi(T^{n}x))$ 
	
	$ n \in \mathbb{N} \cup \{0\} $
	\begin{center}
		$ \Rightarrow F(d(T^{n}x,T^{n+1}x), \varphi(T^{n+1}x),\varphi(T^{n}x))  \preceq (\dfrac{\alpha+\gamma}{1-\beta})^{n} F(d(x,Tx),\varphi(x), \varphi(Tx)) $ $ n \in \mathbb{N} \cup \{0\} $
	\end{center}
	By $ F1 $ we have 
	
	$ max\{d(T^{n+1}x,T^{n}x),\varphi(T^{n+1}x)\} \preceq(\dfrac{\alpha+\gamma}{1-\beta})^{n} F(d(Tx,x),\varphi(Tx), \varphi(x)) $ $ n \in \mathbb{N} \cup \{0\}  $ $ (29) $
	
	$ \Rightarrow d(T^{n+1}x,T^{n}x)\preceq (\dfrac{\alpha+\gamma}{1-\beta})^{n} F(d(Tx,x),\varphi(Tx), \varphi(x)) $$ n \in \mathbb{N} \cup \{0\}  $
	
	Since $ \dfrac{\alpha+\gamma}{1-\beta}\in (0,1) $ which implies that $ \{T^{n}x\} $ is a Cauchy sequence. By the completeness of $ (X,\mathbb{A},d )$, there exists a $ z\in X $ such 
	
	$ lim_{n\rightarrow \infty}d(T^{n}x,z)= \theta $ $ (30) $
	
	From $ (19) $ we have $ lim_{n\rightarrow \infty}\varphi(T^{n+1}x) = \theta$ $ (31) $
	
	Using that $ \varphi  $ is lower semi -continuous, and  from $ (30) $ and $ (31) $ we obtain $ \varphi(z)= \theta $
	
	Then
	
	$ F(d(T^{n+1}x,Tz), \varphi(T^{n+1}x),\varphi(Tz)) \preceq k F(d(T^{n}x,z),\varphi(T^{n}x),\varphi(z))$  , $ n\in \mathbb{N} \cup \{0\} $ $ (32) $
	
	Letting $ n\rightarrow \infty $ in $ (32) $ , using $  (30)$ , $ (31)  $ , $ F2 $ and the continuity of $ F $ we have 
	
	$ F(d(z,Tz),\theta, \varphi (Tz)) \preceq \dfrac{\alpha+\gamma}{1-\beta} F(\theta,\theta,\theta)= \theta $.
	
	From $ F1 $ we obtain $ d(z,Tz)=\theta $ i.e $ z $ is a $ \varphi- $ fixed point of $ T $.
	
	Let $ v \in X $ another $ \varphi- $ fixed point of $ T $.
	
	Putting $ x=z $ and $ y= v $ in $ (26) $ we have
	\begin{center}
		$ F(d(z,v),\theta, \theta)) \preceq \alpha F(d(z,v),\theta,\theta)$  
		
		$ \Rightarrow d(z,v)= \theta$
		
			$ \Rightarrow z=v $    so we get $ (ii) $
	\end{center}	
	\end{proof}
\begin{remark} Taking $ \varphi\equiv \theta $ and $ F(a,b,c)= a+b+c $ in Theorem $ 3.17 $ we obtain Reich's fixed point theorem. 
\end{remark}
\begin{theorem} Let $ (X,\mathbb{A},d) $ be a complete  $C^{\ast}$-algebra valued metric space,
	
	$ \varphi :X \rightarrow \mathbb{A}^{+} $ be a function, and $ F \in \mathcal{F} $.
	
	Suppose that the following condictions hold:
	\begin{itemize}
		\item [(a)] $ \varphi $ is lower semi-continous
		\item[(b)] $ T: X \rightarrow X $ satisfying 
		
		$ F(d(Tx,Ty),\varphi(Tx),\varphi(Ty)) \preceq k[ F(d(x,Ty),\varphi(x),\varphi(Ty))-F(\theta,\varphi(x),\varphi(Ty)) +F(d(y,Tx),\varphi(y),\varphi(Tx))] $  $ k\in (0,\frac{1}{2})$ $ \forall x,y \in X $ $ (33) $
		\end{itemize}
	Then 
	\begin{itemize}
		\item [(i)] $ F_{T} \subseteq Z_{\varphi} $
		\item[(ii)] $ T $ is $ \varphi- $ Picard operator
	 	\end{itemize}
	\end{theorem}
\begin{proof}Let $ u\in X $ be a fixed point of $ T $. Applying $ (33) $  with $ x=y=u $,
	
	 we obtain
	
	$ F(\theta,\varphi(u),\varphi(u)) \preceq k[ F(\theta, \varphi(u),\varphi(u))-F(\theta, \varphi(u),\varphi(u))+ F(\theta, \varphi(u),\varphi(u))]$ $ (k\in (0,\dfrac{1}{2})) $
	
	$ \Rightarrow F(\theta, \varphi(u),\varphi(u))= \theta $  
	
	From $ (F1) $ we have	$ \varphi(u) \preceq F(\theta ,\varphi(u),\varphi(u)) $ then $ \varphi(u)=\theta $.
	
		Let $ x\in X $ be an arbitrary point, we have 
	
	$ F(d(T^{n+1}x,T^{n}x), \varphi(T^{n+1}x),\varphi(T^{n}x)) $ 
	
	 $\preceq k [ F(d(T^{n}x,T^{n}x),\varphi(T^{n}x),\varphi(T^{n}x))-F(\theta, \varphi(T^{n}x),\varphi(T^{n}x)) $ 
	 
	  $+ F(d(T^{n-1}x,T^{n+1}x) , \varphi(T^{n-1}x),\varphi(T^{n+1}x))]$ 
		$ n \in \mathbb{N} \cup \{0\} $
	
	$ F(d(T^{n+1}x,T^{n}x), \varphi(T^{n+1}x),\varphi(T^{n}x))$ 
	
	 $ \preceq k [  F(d(T^{n-1}x,T^{n+1}x) , \varphi(T^{n-1}x),\varphi(T^{n+1}x))]$ 
	\begin{center}
		$ \Rightarrow F(d(T^{n+1}x,T^{n}x), \varphi(T^{n+1}x),\varphi(T^{n}x))  \preceq k^{n} F(d(Tx,x),\varphi(Tx), \varphi(x)) $ $ n \in \mathbb{N} \cup \{0\} $
	\end{center}
	By $ F1 $ we have 
	
	$ max\{d(T^{n+1}x,T^{n}x),\varphi(T^{n+1}x)\} \preceq k^{n} F(d(Tx,x),\varphi(Tx), \varphi(x)) $ $ n \in \mathbb{N} \cup \{0\}  $ $ (34)$
	
	$ \Rightarrow d(T^{n+1}x,T^{n}x)\preceq k^{n} F(d(Tx,x),\varphi(Tx), \varphi(x)) $$ n \in \mathbb{N} \cup \{0\}  $
	
	Since $ k\in (0,\dfrac{1}{2}) $ wich implies that $ \{T^{n}x\}\ $ is a Cauchy sequence. By the completeness of $ (X,\mathbb{A},d )$, there exists a $ z\in X $ such 
	
	$ lim_{n\rightarrow \infty}d(T^{n}x,z)= \theta $ $ (35) $
	
	From $ (35) $ we have $ lim_{n\rightarrow \infty}\varphi(T^{n+1}x) = \theta$ $ (36) $
	
	Using that $ \varphi  $ is lower semi-continuous, and  from $ (35) $ and $ (36) $ we obtain $ \varphi(z)= \theta $
	
	Then
	
	$ F(d(T^{n+1}x,Tz), \varphi(T^{n+1}x),\varphi(Tz)) \preceq k[ F(d(T^{n+1}x,T^{n}x),\varphi(T^{n+1}x),\varphi(T^{n}x))+F(d(Tz,z),\varphi(Tz),\varphi(z))$  , $ n\in \mathbb{N} \cup \{0\} $ $ (37) $
	
	Letting $ n\rightarrow \infty $ in $ (37) $ , using $  (35)$ , $ (36)  $ , $ F2 $ and the continuity of $ F $ we have 
	
	$(1-k) F(d(Tz,z),\theta, \varphi (Tz)) \preceq k F(\theta,\theta,\theta)= \theta $.
	
	From $ F1 $ we obtain $ d(z,Tz)=\theta $ i.e $ z $ is a $ \varphi- $ fixed point of $ T $.
	
	Let $ v \in X $ another $ \varphi- $ fixed point of $ T $.
	
	Putting $ x=z $ and $ y= v $ in $ (33) $ we have
	\begin{center}
		$ F(d(z,v),\theta, \theta)) \preceq k F(\theta,\theta,\theta) $  
		
		$ \Rightarrow d(z,v)= \theta$
		
		$ \Rightarrow z=v $    so we get $ (ii) $
	\end{center}	
\end{proof}
\begin{remark} Taking $ \varphi\equiv\theta $ and $ F(a,b,c)= a+b+c $ in Theorem $ 3.19 $ we obtain Chatterjea's fixed point theorem. 
\end{remark}
\section{Applications}
We deduce in this section some fixed point theorems in $ C^{\ast} $ algebra valued partial metric spaces.

If we consider the $ C^{\ast} $ algebra valued metric $ p^{s} : X \times X \rightarrow \mathbb{A} $ defined by

$ p^{s}(x,y)= 2p(x,y)-p(x,x) -p(y,y)$  and the function $ \varphi: X \rightarrow \mathbb{A} $ defined by 

$ \varphi(x)= p(x,x) $ with $ F(a,b,c)= a+b+c $ $ \forall a,b,c \in \mathbb{A} $

From Theorem $ 3.7 $ we obtain the following result.
\begin{corollary} Let $ (X,\mathbb{A},p) $ be a complete $C^{\ast}$-algebra valued partial metric space and $ T:X\times X \rightarrow $ be a mapping such that 
	\begin{center}
	$ p(Tx,Ty)\preceq k p(x,y) $ , $\forall x,y \in X$,  $ k\in (0,1)$
	\end{center}
Then $ T $ has a unique fixed point $ u \in X $, such that $ p(u,u)= \theta $.
\end{corollary}
From Theorem $ 3.11 $ we have the following result.
\begin{corollary} Let $ (X,\mathbb{A},p) $ be a complete  $C^{\ast}$-algebra valued partial metric space and $ T:X\times X \rightarrow $ be a mapping such that 
	\begin{center}
		$ p(T^{2}x,Tx)\preceq k p(Tx,x) $ , $\forall x \in X$ ,  $ k\in (0,1)$
	\end{center}
	Then $ T $ has a unique fixed point $ u \in X $ , such that $ p(u,u)= \theta $.
\end{corollary}
similary, from Theorem $ 3.13 $ we obtain the following result.
\begin{corollary} Let $ (X,\mathbb{A},p) $ be a complete $C^{\ast}$-algebra valued partial metric space and $ T:X\times X \rightarrow $ be a mapping such that 
	\begin{center}
		$ p(Tx,Ty)\preceq k p(x,y) +\alpha (p(y,Tx)-\dfrac{p(y,y)+p(Tx,Tx)}{2})$ , $\forall x,y \in X$ ,  $ k\in (0,1)$ and $ \alpha \geq 0 $
	\end{center}
	Then $ T $ has a unique fixed point $ u \in X $ , such that $ p(u,u)= \theta $.
\end{corollary}
From Theorem $3. 15 $ we have the following Kannan's fixed point theorem on $C^{\ast}$-algebra valued partial metric space.
\begin{corollary} Let $ (X,\mathbb{A},p) $ be a complete  $C^{\ast}$-algebra valued partial metric space and $ T:X\times X \rightarrow $ be a mapping such that 
	\begin{center}
		$ p(Tx,Ty)\preceq k [p(x,Tx) + (p(y,Ty)]$ , $\forall x,y \in X$ ,   $ k\in (0,\frac{1}{2})$
	\end{center}
	Then $ T $ has a unique fixed point $ u \in X $ , such that $ p(u,u)= \theta $.
\end{corollary}
From Theorem $ 3.17 $ we get the following Reich's fixed point theorem on $C^{\ast}$-algebra valued partial metric space.
 \begin{corollary} Let $ (X,\mathbb{A},p) $ be a complete $C^{\ast}$-algebra valued partial metric space and $ T:X\times X \rightarrow $ be a mapping such that 
	\begin{center}
		$ p(Tx,Ty)\preceq \alpha p(x,y)+ \beta  p(x,Tx) + \gamma p(y,Ty)$ , $\forall x,y \in X$ ,  $ \alpha ,\beta , \gamma \in [0, \infty)$  with  $ \alpha +\beta +\gamma < 1 $
	\end{center}
	Then $ T $ has a unique fixed point $ u \in X $ , such that $ p(u,u)= \theta $.
\end{corollary}
From Theorem $3.19 $ , we obtain the following Chatterjea's fixed point theorem on $C^{\ast}$-algebra valued partial metric space. 
\begin{corollary} Let $ (X,\mathbb{A},p) $ be a complete $C^{\ast}$-algebra valued partial metric space and $ T:X\times X \rightarrow $ be a mapping such that 
	\begin{center}
		$ p(Tx,Ty)\preceq k [p(x,Ty) + (p(y,Tx)]$ , $\forall x,y \in X$ ,  $ k\in (0,\frac{1}{2})$
	\end{center}
	Then $ T $ has a unique fixed point $ u \in X $ , such that $ p(u,u)= \theta $.
\end{corollary}
\section{Acknowledgments}
It is our great pleasure to thank the referee for his careful reading of the paper and for several helpful suggestions.
\bibliographystyle{amsplain}

\end{document}